\documentclass[11pt,leqno,twoside]{amsart}
\usepackage{mathpazo}
\usepackage{geometry}
\usepackage{tikz-cd}
\usepackage{graphicx} 
\usepackage{amsmath}
\usepackage{amssymb}
\usepackage{extarrows}
\usepackage{csquotes}
\usepackage{comment}
\usepackage{mathrsfs}
\usepackage{amsthm}
\usepackage{enumitem}
\usepackage{color}
\usepackage{tikz}
\usepackage[colorlinks=true, linkcolor=blue, urlcolor=blue, citecolor=blue]{hyperref}
\newtheoremstyle{plainsl}
{8pt plus 2pt minus 4pt}
{8pt plus 2pt minus 4pt}
{\slshape}
{0pt}
{\bfseries}
{.}
{5pt plus 1pt minus 1pt}
{}
\theoremstyle{plainsl}
\newtheorem{theorem}{Theorem}[section]

\newtheorem{prop*}{Proposition}

\newtheorem{lemma}[theorem]{Lemma}
\newtheorem{corollary}[theorem]{Corollary}

\newtheorem{remark}[theorem]{Remark}
\theoremstyle{definition}
\newtheorem{definition}[theorem]{Definition}

\def\ba#1\ea{\begin{align*}#1\end{align*}}

\newcommand{\C}{\mathbb{C}}
\newcommand{\Z}{\mathbb{Z}}
\newcommand{\Q}{\mathbb{Q}}
\newcommand{\R}{\mathbb{R}}

\newcommand{\bD}{\mathbf{D}}

\newcommand{\sA}{\mathscr{A}}

\newcommand{\sR}{\mathscr{R}}

\title[Euler characteristic operator]{Triviality of \(K\)-homology classes of Euler characteristic operators of  stratified pseudomanifolds}
\author{Mohammadali Aligholi}
\address{M. Aligholi: Department of Mathematics, University of Wisconsin-Madison, 480 Lincoln Drive, Madison WI 53706-1388, USA}
\email{aligholi@wisc.edu}
\date{}
\begin{document}
\raggedbottom

\begin{abstract}

We show that the $K$-homology classes of the Euler characteristic operators associated with the minimal and maximal ideal boundary conditions of the de Rham differential on the regular part of a connected compact Thom–Mather stratified pseudomanifold, equipped with an adapted metric, are trivial beyond their indices. These indices are the Euler characteristics of the corresponding $L^2$-cohomologies. This generalizes a result of J. Rosenberg for manifolds. The proof uses Witten deformation of the de Rham operator associated with a stratified Morse function, following the work of Álvarez López and Calaza.

\end{abstract}
\subjclass[2020]{19K33, 19K35, 55N33, 58A12, 58A35}
\keywords{de Rham operator, Euler characteristic, analytic $K$-homology, stratified Morse function}
\setcounter{tocdepth}{3}
\maketitle
\section{Introduction}

Let $M$ be a closed oriented real manifold.   The de Rham operator $D=d+d^*$ acting on differential forms of $M$ is a formally self-adjoint operator and is essentially self-adjoint on the Hilbert space of $L^2$ sections of differential forms on $M$. There are two natural gradings on the bundle of differential forms: with the grading by parity of degree, denoted by $\epsilon$, $D$ becomes the Euler characteristic operator, and with the nontrivial signature grading, denoted by $\tau$,  defined by the Hodge $\star$ operator, $D$ is the signature operator when $M$ is even-dimensional.

The indices of these two operators are two essential topological invariants of $M$, respectively the Euler characteristic $\chi(    M)$, which can be defined as the alternating sum of Betti numbers of $M$, and the signature of $M$, defined as the signature of the pairing induced by Poincaré duality on singular cohomology (assuming $\dim M$ is divisible by 4, the pairing in middle-dimensional cohomology is symmetric; in other dimensions the signature is zero). The signature is richer than the Euler characteristic as it remembers some data of the Poincaré pairing. The signature can be lifted to characteristic classes called L-classes, which historically are used in surgery theory and classification of higher-dimensional manifolds by Browder and Novikov, and the famous Novikov conjecture concerns those.

By Kasparov theory, the Euler characteristic and signature operator give rise to $K$-homology classes, $$[D^{dR}_M]\in K^{an}_0(M)=KK_0(C_0(M),\C),\quad [D^{sign}_M]\in K^{an}_{n}(M)=KK_n(C_0(M),\C)$$ where $n=\dim M$ mod 2.  There is a homological Chern character (which is defined for any finite CW complex $M$)
$$Ch: K^{an}_*(M)\rightarrow H_{*}(M,\mathbb{Q})$$
When $M$ is compact and connected, its degree-zero component
$Ch_0([A])\in H_0(M,\mathbb{Q})\cong\mathbb{Q}$,
for a class represented by a graded Dirac operator $A$,
is identified with its graded index. By $K$-theoretic Atiyah-Singer theory \cite{A,AS3}, the Chern character of $[D^{sign}_M]$ is the Poincare dual of cohomological Atiyah-Singer L-classes, which differs from the usual Hirzebruch homological L-classes by powers of 2. Interestingly, the $K$-homology class of the Euler characteristic operator $[D^{dR}_M]$ is trivial beyond its index.
\begin{theorem} (Rosenberg \cite{R})
    Let $M$ be a connected manifold (either compact or non-compact). The analytic $K$-homology class of $D_M^{dR}$, defined by any complete Riemannian metric, is trivial beyond its index, in the sense that it is in the image of $\mathbb{Z}\cong K^{an}_0(pt)\rightarrow K^{an}_0(M)$ (induced by the inclusion of a basepoint into $M$).
\end{theorem}

Rosenberg's theorem is formulated in real $KO$-homology. Here we work with complex $K$-homology.

The goal of this paper is to prove the analogue of the above result for compact Thom–Mather stratified pseudomanifolds.

The analytic study of pseudomanifolds started with the foundational work of Cheeger \cite{Ch} by equipping the regular part with a piecewise flat conical metric. In this paper we use an \textit{adapted metric} as in the terminology of  \cite{BHS,AC} (Definition \ref{adapteddef}), which can be put on any stratum of a Thom-Mather stratified space. Such metrics are quasi-isometric to metrics of piecewise-flat conic type on the regular part considered by Cheeger \cite{Ch}.

On a pseudomanifold there are different (Hilbert-complex) closed extensions  of  de Rham differential $d$, called ideal boundary conditions, initially defined as an unbounded operator on the Hilbert space of  $L^2$  differential forms on the regular part $\mathcal{H}:=L^2(\Omega^*(X_{reg}))$. There are minimal and maximal extensions of $d$, denoted by $d_{min}, d_{max}$.  Then we may form de Rham operators $$D_{rel}:=d_{min}+(d_{min})^*,\quad D_{abs}:=d_{max}+(d_{max})^* $$ Equipping $\mathcal{H}$ with the parity-of-degree grading, $D_{rel},D_{abs}$ are  odd operators, which are also called Euler characteristic operators.  They give rise to $K$-homology classes in $K^{an}_0(X)$ denoted by $[D^{dR}_{X,rel}]$ and  $[D^{dR}_{X,abs}]$.

\begin{theorem}\label{definingclassKhomology}
Let $X$ be a  compact Thom–Mather stratified pseudomanifold. The Euler characteristic operators $D_{rel/abs}$ of $X$ endowed with an adapted metric, acting on $L^2(\Omega^*(X_{reg}))$ equipped with the usual parity grading, define classes $[D^{dR}_{X,rel}],[D^{dR}_{X,abs}]\in K^{an}_0(X)$. These classes do not depend on the choice of the  adapted metric.
\end{theorem}

We note that if $X$ satisfies the Witt condition, then $d_{min}=d_{max}$, and hence $D_{rel}=D_{abs}$. In this case if, in addition, $X$ is oriented and even-dimensional and $\mathcal{H}$ is equipped with the signature grading $\tau$, the common operator $D_{rel}=D_{abs}$ defines the signature operator and gives rise to a $K$-homology class $[D^{sign}_X]\in K_{0}^{an}(X)$ whose Chern character differs from the Goresky–MacPherson L-class by certain powers of $2$; see \cite{ALMP,MW}.

\begin{theorem}\label{maintheorem}
    Let $X$ be a connected compact  Thom–Mather stratified pseudomanifold. Then
    $$[D^{dR}_{X,rel}], [D^{dR}_{X,abs}]\in K^{an}_0(X)$$ are in the image of $ K^{an}_0(pt)\rightarrow K^{an}_0(X)$ induced by the inclusion of a  point of $X$.
\end{theorem}

The graded indices of $D^{dR}_{X,rel/abs}$ are the Euler characteristics of  Hilbert complexes   $$\mathcal{L}^{\bullet}_{min/max}X:=\left(L^2(\Omega^\bullet(X_{reg})),d_{min/max}\right)$$ respectively. When $X$ is oriented, we have
\[H^*(\mathcal{L}^{\bullet}_{min}X)\cong IH^{*}_{\overline{m}}(X)\]
\[H^*(\mathcal{L}^{\bullet}_{max}X)\cong IH^{*}_{\underline{m}}(X)\]
 where   $\underline{m}, \overline{m}$ are lower-middle and upper-middle perversities, and  intersection cohomology is defined by dualizing the intersection-chain complex as in \cite{Be}. This is a result of Cheeger \cite{Ch} in the case of conic metrics. For a general compact  stratified pseudomanifold and a more general family of metrics (a weighted version of adapted metrics),  see a result of F. Bei \cite[Thm.~4]{Be}.

\begin{corollary} Let $X$ be a connected compact oriented Thom–Mather stratified pseudomanifold. Then
    \[Ch([D^{dR}_{X,rel}])=Ch_0([D^{dR}_{X,rel}])=index([D^{dR}_{X,rel}])= \chi(IH^{*}_{\overline{m}}(X))\in H_0(X,\Q)\]
      \[Ch([D^{dR}_{X,abs}])=Ch_0([D^{dR}_{X,abs}])=index([D^{dR}_{X,abs}])=\chi(IH^{*}_{\underline{m}}(X))\in H_0(X,\Q)\]
    where   $\underline{m}, \overline{m}$ are lower-middle and upper-middle perversities, $\chi$ denotes  Euler characteristic and is understood as $\chi[pt]\in H_0(X)$.
\end{corollary}

\begin{corollary}
    Let $X$ be a connected compact oriented Thom–Mather stratified  Witt pseudomanifold, then $d_{min}=d_{max}$ hence $D_{rel}=D_{abs}$ and we denote  $[D^{dR}_{X}]=[D^{dR}_{X,abs}]=[D^{dR}_{X,rel}]$, then
    \[Ch([D^{dR}_{X}])=index([D^{dR}_{X}])=I\chi(X)\]
    where $I\chi(X)$ is the  Euler characteristic of intersection cohomology of $X$.
\end{corollary}

For the proof of Theorem \ref{maintheorem} we follow Rosenberg \cite{R} in the stratified setting.  In loc. cit., Rosenberg considers a deformation of the de Rham operator $D_{t}:=D+t A$ where $A$ is a zero-order operator defined by Clifford multiplication by a vector field. This idea goes back to Witten \cite{W} and is known as Witten deformation. When the vector field is the gradient of a real-valued function $f$ the deformation takes the form, 
 \[d_t:=e^{-tf}de^{tf}=d+tdf\wedge\]
   \[D_t:=d_t+d^*_t=(d+d^*)+t(df\wedge+\iota_{\nabla f})=D+t\hat{c}(df)\]
   where  $t>0$ is a parameter.

 The point of this deformation is that when $f$ satisfies some boundedness assumptions the de Rham complex of (minimal and maximal extensions of) $d_t$ is isomorphic to that of $d$, and the K-homology class of $D_t$ is the same as that of $D$, but as $t\rightarrow \infty$, $D_t$ develops a growing spectral gap away from the critical points.   We choose $f$ to be  a \textit{stratified Morse} function (Definition \ref{stratmorsedef}) in the sense of \cite{AC}, which satisfies the required boundedness assumptions.
 
 The relation between the Euler characteristic of a manifold and a Morse function is old and goes back to Morse \cite{M}; Morse established the inequalities relating the numbers of critical points of each index to the Betti numbers.

Morse theory in the stratified setting was originally considered by Goresky--MacPherson \cite{GM} and later studied by others including \cite{AC}. Different conditions can be imposed on a stratified Morse function. In this paper we follow the work of Álvarez López--Calaza \cite{AC}, where they use the term \textit{rel-Morse functions}; these are different from those considered in \cite{GM}. Such a stratified Morse function restricts to a smooth function on $X_{reg}$, with bounded derivatives and Hessians, and has isolated critical points. Witten deformation of the de Rham operator by a stratified Morse function was studied by Ursula Ludwig and Álvarez López--Calaza; see \cite{L,AC}.

The steps of the proof of Theorem \ref{maintheorem}  are as follows. Using the spectral gap of $D_{t,rel/abs}$ away from the critical points, we show that the restriction of $[D^{dR}_{X,rel/abs}]$ to the complement of a union of sufficiently small compact stratified neighborhoods of critical points vanishes. Then by the long exact sequence of K-homology, the class $[D^{dR}_{X,rel/abs}]$ comes from the neighborhoods, which can be chosen to be contractible.

For completeness we mention that the local contribution of (rel-)critical points to Euler characteristic is given explicitly in \cite[Theorem~1.2]{AC} as 
\[\chi(\mathcal{L}^{\bullet}_{min/max}X)=\sum_{r\in \Z} (-1)^r \left( \sum_{x\in~ critical(f)~in~\hat{X}_{reg}} \nu^r_{x,min/max}\right)\]
 where $\nu^r_{x,min/max}$
is a number which can be expressed in terms of the \(L^2\)-Betti numbers of the links appearing in the local model of  $f$ around the rel-critical point $x$ (see \cite[Def.~4.8]{AC}). Here $\hat{X}_{reg}$ is the metric completion of $X_{reg}$ and rel-critical points of $f$,  in the sense of loc. cit.,  are points in this space, but there is a finite surjective continuous map $\hat{X}_{reg} \rightarrow X$, and by the critical points of $f$ on $X$ we mean the images of the rel-critical points under this map.

\medskip
{\bf Acknowledgments.}  I would like to thank Pierre Albin for patiently answering my various questions over the course of the past year.

\section{de Rham operators}
Let $X$ be a compact  Thom–Mather stratified pseudomanifold (e.g., see \cite{ALMP} for definitions of these notions in detail). Let $X_{\mathrm{reg}}$ denote the dense regular part, the union of the top-dimensional strata. Note that we don't require the Witt condition. In this paper  we equip the regular part $X_{reg}$ of $X$ with an \textit{adapted} metric $g$.

\begin{definition} \cite{AC, BHS} \label{adapteddef} Let $X$ be a Thom–Mather stratified pseudomanifold.
An adapted metric $g$ on $X_{reg}$ is defined inductively on the depth of the stratification of $X$. If depth $X$ = 0, then $X$ is a
manifold, and any Riemannian metric on $X$ is called adapted. A neighborhood of  a point $x\in X\setminus X_{reg}$ is of the form $V=B\times C_{<\epsilon}(L)$ centered at $x$ where $B\cong \R^n$ is an open ball ($n$ is the dimension of the stratum that contains $x$), and $L$ is a stratified space of lower dimension (link), 
then $X_{reg}\cap V$  is of the form $B\times L_{reg} \times (0,\epsilon)$ where $L_{reg}$ is the regular part of $L$. On this neighborhood
$g$ is quasi isometric to the model  metric
\[g_{\R^n}+r^2g_{L}+dr^2\]
where $g_{\R^n}$ is the Euclidean metric on $B$ , $g_L$ is an adapted metric on $L_{reg}$, and $r$ is the radial coordinate on the cone.
\end{definition}
For the existence of adapted metrics on Thom-Mather stratified spaces see \cite{BHS}.

\begin{remark}
A metric quasi-isometric to a piecewise flat conic metric, in the sense of \cite{Ch}, and an iterated edge (wedge) metric  in the sense of \cite{ALMP} is an  adapted metric.
\end{remark}

\begin{lemma}\cite{AC,BHS}
    If $X$ is compact, any two adapted metrics on $X_{reg}$ are (globally) quasi-isometric.
\end{lemma}

\begin{remark}
Any two  adapted metrics
are homotopic, in fact the space of such metrics is convex;  if $g_0,g_1$ are such metrics then for $0\leq t\le 1, tg_0+(1-t)g_1$ is also  an adapted metric.  
\end{remark}

We give precise definitions of de Rham operators. 
Let $\Omega^*(X_{reg})$ be the (complex-valued) smooth differential forms on $X_{reg}$ with arbitrary support. Let $$\mathcal{H}:=L^2(\Omega^*(X_{reg}))$$ be  the Hilbert space of $L^2$ differential forms on the possibly incomplete Riemannian manifold $(X_{reg},g)$. The de Rham differential $d$ is an unbounded operator initially defined on $\Omega_{c}^*(X_{reg})\subset \mathcal{H}$, the smooth differential forms with compact support on $X_{reg}$.
 Similarly we have codifferential $\delta$ which is the formal adjoint of $d$. Then the minimal domain is defined as 
\begin{align*}
    Dom({d}_{min})= \overline{\Omega_{c}^*(X_{reg})}
    = \{w\in \mathcal{H}|\exists w_j\in \Omega_{c}^*(X_{reg}), \eta \in \mathcal{H}, s.t.~ w_j\rightarrow w, dw_j\rightarrow \eta, in~ \mathcal{H}\}
\end{align*}
where $ \overline{\Omega_{c}^*(X_{reg})}$ denotes the graph closure, and in this case $d_{min}w:=\eta$.
The maximal domain is defined as
\begin{align*}
    Dom({d}_{max})=&\{w\in \mathcal{H}|dw \in \mathcal{H}~ distributionally\} \\
    =&\{w\in \mathcal{H}|\exists \eta \in \mathcal{H},~ \langle \eta,\ \gamma \rangle=\langle w,\delta\gamma\rangle ~ \forall \gamma \in \Omega_{c}^*(X_{reg}) \}
\end{align*}
in this case $d_{max}w:=\eta$. We have $Dom(d_{min})\subset Dom(d_{max})$,
and in general these domains are not equal. One can also define ${\delta_{min}}, \delta_{max}$ similarly. 

The following properties, which are invariant under quasi-isometries, hold for any adapted metric. These are essentially proved by Cheeger but also see \cite{AC,Be,ALMP}. See \cite[Thm.~1.1]{AC} for $(5)$ in the case of adapted metric.

\begin{theorem}\cite{Ch,Ch1,AC}\label{proptiesanalytic}
 Let $X$ be compact Thom–Mather stratified  pseudomanifold endowed with an adapted metric. Then
\begin{enumerate}
    
    \item 
\[(d_{max})^*={\delta}_{min}\]
\[(d_{min})^*={\delta}_{max}\]

    \item  If $X$ satisfies the Witt condition then $d_{min}=d_{max}$, i.e., $Dom(d_{min})= Dom(d_{max})$.
\item  The operators
    \[D_{rel}:=d_{min}+(d_{min})^*=d_{min}+\delta_{max}\]
\[Dom(D_{rel})=Dom(d_{min})\cap Dom(\delta_{max})\]
\[D_{abs}:=d_{max}+(d_{max})^*=d_{max}+\delta_{min}\]
\[Dom(D_{abs})=Dom(d_{max})\cap Dom(\delta_{min})\]
are self-adjoint. \\

\item As operators on their domain endowed with the graph norm, $D_{rel},D_{abs}$ are Fredholm.\\

\item  $D_{rel}$ and $D_{abs}$ have discrete spectra with finite multiplicities.\\

\end{enumerate}
\end{theorem}
\begin{remark} In \cite{AC}, the notation $D_{min/max}$ is  used instead of $D_{rel/abs}$, we avoid this notation because it may be confusing; $D_{rel}$ need not coincide with the minimal extension of $D=d+\delta$ originally defined on $\Omega_{c}^*(X_{reg})$, and $D_{abs}$ need not coincide with the maximal extension of $D$. 
\end{remark}

\begin{remark}
    For an arbitrary adapted metric, $D$ need not be essentially self-adjoint, i.e., minimal and maximal extensions of $D$ may not coincide,  even in the Witt case, as essential self-adjointness is not quasi-isometric invariant. There exists an adapted metric, such as an \textit{adapted rigid iterated edge metric} in the sense of \cite[Prop.~5.4]{ALMP} in the Witt case which is obtained by rescaling of the link metric of a wedge metric that satisfies a link spectral gap condition, such that $D$ becomes essential self-adjoint \cite[Thm.~1.1]{ALMP}, hence the self-adjoint extension is unique.
\end{remark}

\subsection{Kasparov cycles and  $K$-homology }

The unbounded formalism for the Kasparov KK-group \cite{BJ} is as follows. Let $X$ be a second-countable locally compact Hausdorff space.

\begin{definition}\label{kkdef}
    An even unbounded $K$-cycle representing a class in  $K_0^{an}(X)=KK_0(C_0(X),\C)$ consists of a tuple $(\mathcal{H},\Gamma, \rho_X, \bD)$
where $\mathcal{H}$ is a separable  $\mathbb{Z}_2$ graded Hilbert space graded by $\Gamma$ such that $\Gamma^2=1, \Gamma^*=\Gamma$, and $\rho_X:C_0(X)\rightarrow \mathcal{B}(\mathcal{H})$ is an even $\ast$ representation and $\bD:Dom(\bD)\subset\mathcal{H}\rightarrow\mathcal{H}$ is an odd  self-adjoint unbounded operator with respect to $\Gamma$ such that 
\begin{enumerate}
    \item There exists a dense $\ast$-subalgebra $A\subset C_0(X)$ such that $Dom(\bD)$ is invariant under multiplication by elements of $A$ and $[\bD,\rho_X(f)]$ extends to a bounded operator for every $f\in A$.
    \item For every $f\in C_0(X),$ \[\rho_X(f)(1+\bD^2)^{-1}\] is compact.
\end{enumerate}
When $X$ is compact and $\rho_X$ is unital, choosing $f=1\in C_0(X)$, the second condition implies and is equivalent to that $\bD$ has compact resolvent, which is equivalent to $\bD$ having purely discrete spectrum with finite multiplicities. 
\end{definition}

The bounded description for the Kasparov K-homology is as follows:
\begin{definition}\label{kkboundeddef}
An even  bounded $K$-cycle  representing a class in $K_0^{an}(X)=KK_0(C_0(X),\C)$ consists of a tuple
$   (\mathcal{H},\Gamma,\rho_X,F)
   $
where $\mathcal{H}$ is a separable $\mathbb{Z}_2$-graded Hilbert space graded by $\Gamma$ such that $\Gamma^2=1, \Gamma^*=\Gamma$, $\rho_X:C_0(X)\rightarrow \mathcal{B}(\mathcal{H})$ is an even $\ast$-representation, and $F\in\mathcal{B}(\mathcal{H})$ is an odd bounded operator with respect to $\Gamma$ such that there is a dense $\ast$-subalgebra $A\subset C_0(X)$ such that for every $f\in A$,
\begin{enumerate}

\item $[F,\rho_X(f)]$ is compact.
\item $(F-F^\ast)\rho_X(f)$ is compact.
\item $(F^2-1)\rho_X(f)$ is compact.
\end{enumerate}
If $F$ is self-adjoint, the second condition is automatic. 
\end{definition}

Given an unbounded cycle $(\mathcal{H},\Gamma, \rho_X, \bD)$ one obtains a bounded representative of the same $K$-homology class by taking the
bounded transform
\[
F_{\bD}:=\bD(1+\bD^2)^{-1/2}.
\]
Thus
\[
(\mathcal{H},\Gamma,\rho_X,F_{\bD})
\]
is a bounded $K$-cycle and
\[
[\bD]
=
\left[
\mathcal{H},\Gamma,\rho_X,
\bD(1+\bD^2)^{-1/2}
\right]
\in K_0^{an}(X)
\]

\subsection{Proof of Theorem \ref{definingclassKhomology}}
\begin{proof}
We use the unbounded description of $K$-homology and follow the method of \cite{MW} and \cite[Thm.~6.2]{ALMP}.  Let $\bD=D_{rel}$ or $D_{abs}$.
Here $\mathcal{H}=L^2(\Omega^*(X_{reg}))$.
$\Gamma=\epsilon, \epsilon|_{\Omega^i}=(-1)^i$ so $\bD$ is an odd operator, and $\rho_X$ is the natural representation by multiplication of $C(X)$. For a dense subalgebra $A\subset C(X)$, one can take the subalgebra of Lipschitz functions with respect to the given metric structure on $X$. Multiplication by such functions preserves $Dom(\bD)$ (see the proof of Lemma \ref{lemm1}), and for any Lipschitz function $f$, the commutator $[\bD,f]$ is Clifford multiplication by $df$, which exists almost everywhere and is bounded on $X$. Compactness of the resolvent of $\bD$ follows from Theorem \ref{proptiesanalytic} (5). Invariance under the choice of adapted metric follows
from the abstract operatorial homotopy argument of
\cite[Secs.~6.2--6.3]{Hi}. For two adapted metrics $g_0,g_1$ let  $g_s=(1-s)g_0+sg_1$, $s\in[0,1]$. Since any two
adapted metrics are globally quasi-isometric, the
Hilbert spaces
$\mathcal H_s=L^2(\Omega^*(X_{reg}),g_s)$
can be identified as vector spaces with uniformly
equivalent norms. Moreover, their inner products
vary continuously in operator norm in the sense
of \cite[Sec.~6.2]{Hi}.
The domains of $d_{min}$ and $d_{max}$ are
independent of $s$, by the definitions of these
extensions and equivalence of the corresponding
graph norms. Fix $d=d_{min}$ or $d=d_{max}$.
Then $d$ is a fixed densely defined closed operator
satisfying $d^2=0$.
Let $d_s^*$ denote its adjoint with respect to the
inner product induced by $g_s$, and set
$\bD_s=d+d_s^*$.
By \cite[Prop.~6.1]{Hi}, the family
$s\mapsto(\bD_s+i)^{-1}$ is norm-continuous.
Furthermore, the resolvents are compact by
Theorem \ref{proptiesanalytic}, and the commutators
$[\bD_s,f]$ form a bounded continuous family for
Lipschitz functions $f$ on $X$.
Therefore, the construction of
\cite[Prop.~6.4]{Hi}, with the fixed parity grading,
gives an unbounded operatorial homotopy between
the corresponding Kasparov cycles. Consequently,
$[\bD_0]=[\bD_1]$ in $K^{an}_0(X)$.
\end{proof}

For simplicity of notation, we alternate between $[D^{dR}_{X,abs}],[D^{dR}_{X,rel}]$ and $[D]$ when it is clear from the context which operator we are talking about. We also sometimes write $C(X)$ instead of $\rho_X$ as all the representations of $K$-cycles in this paper are  pointwise multiplication by continuous functions $C(X)$.

\section{Witten deformation of de Rham operator}

Let $f$ be a real valued smooth function on $X_{reg}$ with bounded derivative $|df|_g\leq C$.
On $\Omega_c^*(X_{reg})\subset L^2(\Omega^*(X_{reg}))$ consider unbounded  operators 
\[d_t:=e^{-tf}de^{tf}=d+tdf\wedge\] 
\[\delta_t:=e^{tf}\delta e^{-tf}=\delta+t\iota_{\nabla f} \]
where $t>0$ is a parameter. Note that  $df\wedge, \iota_{\nabla f}$ are zero-order operators which are bounded by the boundedness assumption on $df$; therefore their closures are themselves with maximal domain simply $L^2(\Omega^*(X_{reg}))$. Thus minimal and maximal extensions are 
\[({d_{t}})_{min/max}={d}_{min/max}+tdf\wedge\]
\[({\delta_{t}})_{max/min}=(({d_{t}})_{min/max})^*=(d_{min/max})^*+t\iota_{\nabla f}\]
Define
\[A:=df\wedge+\iota_{\nabla f}\]
\[D_{t, rel/abs}:=({d_{t}})_{min/max}+(({d_{t}})_{min/max})^*=(d_{min/max}+(d_{min/max})^*)+tA=D_{rel/abs}+tA\]
therefore
\[Dom(D_{t,rel/abs})=Dom(D_{rel/abs})\]
Note that $[A,g]=0$ for functions $g\in C(X)$ and $A^*=A$.  $A$ is odd with respect to the parity grading, so $D_{t,rel/abs}$ is still an odd operator. 
It's clear that $D_{t,rel/abs}$ defines a class in $K$-homology $[D_{t,rel/abs}]$ just like $D_{rel/abs}$, since $A$ is bounded and of zero order; it doesn't affect the defining analytic conditions of a $K$-cycle.

\begin{theorem} For every $t>0$
   \[[D_{t,rel/abs}]=[D_{rel/abs}]\] 
   \textit{i.e.,  Witten deformation of the de Rham operator by a  smooth function with bounded derivative does not change the $K$-homology class.}
\end{theorem} 
\begin{proof} Let $ D= D_{rel}$ or $D_{abs}$.
    Consider the homotopy
\[D_s=D+stA,\quad 0\leq s\leq 1\]
Then $Dom(D_s)=Dom(D)$. Since $A$ is a bounded zero-order operator, $D_s$ defines a homotopy of $K$-cycles between $D$ and $D_t$.
\end{proof}

\subsection{Stratified Morse function}

From now on we assume $f$ is a \textit{stratified Morse function}.

\begin{definition}\cite{AC}\label{stratmorsedef} Let $X$ be a compact   Thom–Mather stratified pseudomanifold and fix an adapted metric on $X_{reg}$.
    A continuous function $f\in C(X)$ is called a stratified Morse function if
    
\begin{enumerate}
    \item the restriction of $f$ to $X_{reg}$ is a \textit{rel-Morse} function in the sense of \cite{AC}; it's a smooth function such that $f, |df|, Hess(f)$ are all bounded (with respect to the chosen adapted metric), and    there exists a specific local quadratic normal form near every \textit{rel-critical} point.
     \item the \textit{rel-critical}
     points of $f$ are isolated.
\end{enumerate}
\end{definition}
We don't use the local quadratic normal form in this paper, so we refer to \cite{AC} for the precise definition. In the sense of \cite{AC}, a rel-critical point of $f$ is a point $x$ in the metric completion  $\hat{X}_{reg}$ of $X_{reg}$, defined by $$\liminf_{y\rightarrow x, y\in X_{reg}}|df(y)|=0$$
By \cite[Remark~9]{AC}, there is a finite continuous surjective map $\hat{X}_{reg}\rightarrow X$, so we consider the images of rel-critical points of $f$ on $X$; they are isolated points on $X$, and we simply call them the critical points of $f$ on $X$.

By  \cite[Proposition 4.9]{AC},  the subset of functions in $C^{\infty}(X_{reg})$ with continuous extensions to $X$ whose restrictions to all strata of $X$ are rel-Morse functions (with respect to some adapted metric) is dense in the sense of weak $C^{\infty}$ topology. Therefore, stratified Morse functions exist with respect to some adapted metric. Since the $K$-homology class of the Euler characteristic operator is independent of the choice of adapted metric, we may choose and fix an adapted metric for which a stratified Morse function exists.

Boundedness of $Hess(f)$ implies $|df|$ is Lipschitz on $X_{reg}$ and therefore it has a unique continuous extension to its metric completion $\hat{X}_{reg}$, then the set of rel-critical points is the zero set of this extension, hence closed, thus the number of rel-critical points is finite since $\hat{X}_{reg}$ is compact.  Critical points do not necessarily lie in $X_{reg}$, and they can be on lower-dimensional strata. Let $\{p_i\}^k_{i=1}$ be the set of critical points. Choose compact, contractible, sufficiently small distinguished neighborhoods $N_{\epsilon}(p_i)$ such that $p_i$ are in the interiors and they are pairwise disjoint. This is possible by the local picture of a Thom–Mather stratified space; for a point $x\in X$ there is a distinguished open neighborhood
\[B\times c_{<\epsilon}(L)\]
centered at $x$, where $B$ is an open ball and $c_{<\epsilon}(L)$ is a truncated open cone within $\epsilon$ distance of the cone vertex over the link $L$. Now one can take a smaller closed neighborhood $N_{\epsilon'}(x)$ such as
\[\overline{B'}\times c_{\leq\epsilon'}(L)\]
for a smaller ball $\overline{B'}\subset B$ and $\epsilon'<\epsilon$.
Let \[U=X\setminus \bigcup N_{\epsilon}(p_i)\]
then $U$ is an open stratified subset of $X$ where its regular part is 
\[U_{reg}=X_{reg}\setminus \bigcup N_{\epsilon}(p_i)\]

\subsection{Spectral gap away from critical points}
We discuss the main point of Witten deformation, namely the spectral gap away from the critical points.
Consider the Laplacians
\[\Delta_{t,rel/abs}=D^2_{t,rel/abs}\]

For all $t>0$, we have \cite[\S18]{AC}
\begin{equation}
Dom(\Delta_{t,rel/abs})=Dom(\Delta_{rel/abs})
\end{equation}
\begin{equation}\label{idenityhess}
\Delta_{t,rel/abs}=\Delta_{rel/abs}+tHess(f)+t^2|df|^2
\end{equation}
where $Hess(f)$ denotes the endomorphism of differential forms induced by the Riemannian Hessian.
Recall that   $ess~supp(s)\subset U_{reg}$ for a section  $s\in L^2(\Omega^*(X_{reg}))$  means $s=0$ almost everywhere on $X_{reg}\setminus U_{reg}$. Note that in general
\[Dom(\Delta_{rel/abs})=Dom(D_{rel/abs}^2)\not = Dom(D_{rel/abs})\]

\begin{remark}
\textbf{Notation}: For simplicity, throughout this paper we use $\|\cdot\|$ for Hilbert-space and operator norms,
and $|\cdot|$ for pointwise tensor norms.
 For example, for $s\in\mathcal{H}$, $\|s\|$ denotes its $L^2$ norm, while for a bounded operator $F$, $\|F\|$ denotes its operator norm.
\end{remark}

\begin{theorem}\label{specgapmain}(\cite[\S~18.2]{AC}) 
    There exists $C>0, t_0>0$ such that for any section $s\in Dom(D)$  with $ess~supp(s)\subset U_{reg}$ and $t>t_0$
\[||D_{t,rel/abs}s||^2 \geq Ct||s||^2\]
\end{theorem}
\begin{proof}
    For simplicity we drop $rel/abs$ from the notation of $D_{rel/abs}$. For any $s\in Dom(\Delta)$, the identity (\ref{idenityhess}) implies
    \begin{align}\label{ident1}
        ||D_ts||^2=\langle D_ts,D_ts\rangle=\langle \Delta_ts,s\rangle=||Ds||^2+t \langle Hess(f) s,s\rangle +t^2\langle |df|^2 ~s,s\rangle
    \end{align}
    Since $U_{reg}$ is away from the set of critical points (note that $U_{reg}$ is away from critical points from an $\epsilon$-distance), there is a lower bound $m>0$ for $|df|$ on $U_{reg}$. Since $Hess(f)$ is bounded there is $H_0$ such that $||Hess(f)||<H_0$, therefore for $s\in Dom(\Delta)$ with $ess~supp(s)\subset U_{reg}$ we have
    \[||D_ts||^2\geq ||Ds||^2-tH_0||s||^2+t^2m^2||s||^2=||Ds||^2+t||s||^2(-H_0+tm^2)\]
for large $t$, $-H_0+tm^2>C>0$.

    Since $D$ is self-adjoint, $D^2=\Delta$, $Dom(\Delta)$ is dense in $Dom(D)$ in the graph norm. Combined with boundedness of $D_t-D,|df|, Hess(f)$, this implies that the identity (\ref{ident1}) extends continuously to $Dom(D)$,  and the above inequality holds as well.
\end{proof}

See \cite[Proposition 6.24]{J} for this result in the case of a wedge metric.

\section{Analytic estimates}
In this section we prove the estimates needed to prove the vanishing of the restricted $K$-homology class from Theorem \ref{specgapmain}. We drop $rel/abs$ from $D_{rel/abs}$ and $D$ represents any of them.

\begin{lemma}\label{lemm1}
     For $g\in C(X)$ with $supp(g)\subset U$ we have as $t\rightarrow \infty$
    \begin{equation}\label{gboundedgoes0}
        ||g(1+D^2_t)^{\frac{-1}{2}}||\rightarrow 0
    \end{equation}
\end{lemma}
\begin{proof}
    
    For simplicity of notation, let $B_t=(1+D^2_t)^{\frac{-1}{2}}$.
    We have
    \[||B_t||\leq 1,\quad ||D_tB_t||\leq 1\]
    Note that $Dom(B_t)=L^2(\Omega^*(X_{reg}))$.
    Therefore it suffices to show  (\ref{gboundedgoes0}) for Lipschitz functions supported in $U$, because such functions are uniformly dense in the class of continuous functions supported in $U$.
 So assume $g\in Lip(X)$ with $supp(g)\subset U$.

    Let $u\in L^2(\Omega^*(X_{reg}))$, let $v_t:=B_tu$.  $D_tB_t$ is a bounded operator on $L^2(\Omega^*(X_{reg}))$  and $v_t\in Dom(D_t)$. Then 
    \begin{equation} \label{normuv}
            ||v_t||\leq ||u||,\quad ||D_tv_t||=||D_tB_tu||\leq ||u||
    \end{equation}

Recall $D_t=D+tA$ and $[A,g]=0, [D_t,g]=[D,g]$. Moreover $[D,g]$ is bounded as a multiplication operator. Indeed, for a section $s\in \Omega_{c}^*(X_{reg})$
\[d(gs)=gds+dg\wedge s\]
\[\delta(gs)=g\delta s-\iota_{\nabla g}s\]
$$[D,g]s=dg \wedge s -\iota_{\nabla g}s$$
 since $g$ is Lipschitz $dg$ is bounded (note that $dg, \nabla g$ make sense since a Lipschitz function is differentiable almost everywhere on $X_{reg}$). By taking graph-norm approximating sequence and distribution definition, this implies $[D,g]$ is bounded and $g$ preserves domains of $d_{min},d_{max},\delta_{min},\delta_{max}$, hence of  $Dom(D_t)=Dom(D)$ (note that $D=D_{rel}$ or $D_{abs}$). 
 Consider $gv_t=gB_tu$, then since $v_t\in Dom(D)$, $gv_t\in Dom (D_t)$. Then $ess~supp (gv_t)\subset supp(g) \cap X_{reg}\subset  U_{reg}$; therefore by Theorem \ref{specgapmain} we have
\[||gv_t||\leq \frac{||D_t(gv_t)||}{\sqrt{Ct}}\]
We have
\[D_t(gv_t)=gD_tv_t+[D_t,g]v_t=gD_tv_t+[D,g]v_t\]
By (\ref{normuv})
\[||D_t(gv_t)||\leq (||g||_{\infty} + ||[D,g]||)||u||\]
Thus
\begin{equation}
     ||g(1+D^2_t)^{\frac{-1}{2}}u||=||gB_tu||=||gv_t||\leq \frac{(||g||_{\infty} + ||[D,g]||)||u||}{\sqrt{Ct}} 
\end{equation}
So $||g(1+D^2_t)^{\frac{-1}{2}}||\leq \frac{c}{\sqrt{Ct}}$ for some positive constants $C,c$ independent of $t$. The rest is clear.

\end{proof}

Now consider the bounded transform

\[h(x)=\frac{x}{\sqrt{1+x^2}}\]
\[F_t:=h(D_t)=D_tB_t=D_t(1+D^2_t)^{\frac{-1}{2}}\] 
\[1-F^2_t=B^2_t\]
This is the bounded operator corresponding to $D_t$ in the bounded description of $K$-homology.

Let $|A|=(A^2)^{1/2}=|df|$ and
$\widetilde{A}=\frac{A}{|A|}=\frac{df\wedge+\iota_{\nabla f}}{|df|}$.
This is a well-defined self-adjoint operator of operator norm at most $1$ on $\mathcal{H}=L^2(\Omega^*(X_{reg}))$, since the points where $df=0$ have measure zero. It is first defined on the dense subset $\Omega^*_{c}(X_{reg})$ and then extended uniquely to a bounded operator on $\mathcal{H}$.

\begin{lemma}\label{lemma3}As $t\rightarrow \infty$
    \[F_t\rightarrow \widetilde{A}\]
    in the strong operator topology on $L^2(\Omega^*(X_{reg}))$
\end{lemma}
\begin{proof}
The argument of \cite{R} does not apply directly as we don't have a spectral gap on all of the domain of $D_t$, just away from critical points.\\
Consider smooth differential forms with compact support on $X_{reg}$ away from the finitely many critical points of $f$ on $X$, i.e.,
$$\sR:=\Omega^*_{c}(X_{reg}\setminus \{df=0\} )\subset Dom(D)\subset L^2(\Omega^*(X_{reg}))$$
$\sR\subset L^2(\Omega^*(X_{reg}))$ is dense, since 
\[||F_t||\leq 1,~ ||\widetilde{A}||\leq1\]
it suffices to show the strong convergence for a section $w\in \sR$; 
\begin{equation}\label{FTAGOES0}
    ||(F_t-\widetilde{A})w||\rightarrow 0; \quad t\rightarrow \infty
\end{equation}
Let
\[K=supp (w) \Subset (X_{reg}\setminus \{df=0\})\]
There exists $0<l$ such that
\begin{equation}\label{bounddf}
  |df|_{K}\geq l>0
\end{equation}
 It's easier if we make use of resolvents of the operators $D_t,A$. Denote, for $r\in \R\setminus\{0\}$

\[R(r,D_t):=(D_t+ir)^{-1}\]
\[R(r,tA)=(tA+ir)^{-1}\]
Define
\[\sA_t:=h(tA)=tA(1+(tA)^2)^{-1/2}\]
We have the identity 
\begin{equation}\label{identity}
    h(x)=\frac{x}{\sqrt{1+x^2}}=\frac{1}{\pi}\int^{\infty}_0 (\frac{1}{x+i\sqrt{1+\lambda^2}}+\frac{1}{x-i\sqrt{1+\lambda^2}})d\lambda
\end{equation}
The spectral theorem and functional calculus  for self-adjoint operators  \cite[\S5]{S} imply (see also \cite[p.~164]{B})
\begin{equation}\label{integralF_t}
    F_tw=h(D_t)w=\frac{1}{\pi}\int^{\infty}_0 \left(R(\sqrt{1+\lambda^2},D_t)+R(-\sqrt{1+\lambda^2},D_t)\right)w d\lambda
\end{equation}

\[\sA_tw=h(tA)w=\frac{1}{\pi}\int^{\infty}_0 \left(R(\sqrt{1+\lambda^2},tA)+R(-\sqrt{1+\lambda^2},tA)\right)w d\lambda\]
(more precisely, for a bounded Borel measurable function $b$ and a self-adjoint operator $D$, $b(D)=\int_{\R}b(\lambda)dE(\lambda)$, where $E$ is the spectral measure. Then Fubini implies that if $h_u(x)=\int^u_0 k(x,\lambda)d\lambda$ and $k$ is uniformly bounded, then $h_u(D)=\int^u_0 k(D,\lambda)d\lambda$. Now write the improper integral (\ref{identity}) as $h(x)=\lim_{u\rightarrow \infty}h_u(x)$; then $h_u\rightarrow h$ pointwise and $|h_u|\leq 1$, therefore dominated convergence implies $h(D_t)=\lim_{u\rightarrow \infty} h_u(D_t)$ in the strong operator topology.)

Note that because of the support assumption on $w\in \sR$ we have (see (\ref{descrR}) below and the discussion that follows.)
\[R(r,tA)w\in Dom(D)\]
We have the \textit{resolvent identity} \cite[\S2.2]{S}, which can be obtained by basic algebra: (as $D_t-tA=D$)
\begin{equation}\label{resolvide}
    R(r,D_t)w-R(r,tA)w=- R(r,D_t)DR(r,tA)w
\end{equation}
We will prove that for $r\in \R\setminus\{0\}, t>0$
\begin{equation}\label{mainestimate}
    || R(r,D_t)DR(r,tA)w||\leq \frac{C_w}{|r|\max{(t,|r|)}}
\end{equation}
for some constant $0<C_w$ depending on $w$.
Assume (\ref{mainestimate}); then 
\[|| R(r,D_t)w-R(r,tA)w||\leq  \frac{C_w}{|r|\max{(t,|r|)}}\]
Then we have
\begin{align*}
    ||(F_t-\sA_t)w||\leq& 2\frac{1}{\pi}\int^{\infty}_0 \frac{C_w}{(\sqrt{1+\lambda^2})\max{(t,\sqrt{1+\lambda^2})}} d\lambda\\
    =&\frac{2}{\pi}\int^{t}_0 \frac{C_w}{(\sqrt{1+\lambda^2})\max{(t,\sqrt{1+\lambda^2})}} d\lambda+\frac{2}{\pi}\int^{\infty}_t \frac{C_w}{(\sqrt{1+\lambda^2})\max{(t,\sqrt{1+\lambda^2})}} d\lambda\\
    \leq& \frac{2}{\pi}\frac{C_w}{t}\int^{t}_0 \frac{1}{\sqrt{1+\lambda^2}}d\lambda+ \frac{2C_w}{\pi}\int^{\infty}_t \frac{1}{{1+\lambda^2}}d\lambda\\
    \leq & \frac{2C_w\log(3t)}{\pi t}+\frac{2C_w}{\pi t} 
\end{align*}
where we used for $t\geq 1$
\[\int^{t}_0 \frac{1}{\sqrt{1+\lambda^2}}d\lambda=\log(t+\sqrt{1+t^2})\leq  \log(3t)\]
\[\int^{\infty}_t \frac{1}{{1+\lambda^2}}d\lambda\leq \int^{\infty}_t \frac{1}{{\lambda^2}}d\lambda=\frac{1}{t}\]
Therefore as $t\rightarrow \infty$
\[||(F_t-\sA_t)w|| \rightarrow 0\]
To prove (\ref{FTAGOES0}) we have
\[ ||(F_t-\widetilde{A})w||\leq ||(F_t-\sA_t)w|| +||(\sA_t-\widetilde{A})w||\]
so we need to show $||(\sA_t-\widetilde{A})w||\rightarrow 0$. Indeed, as $A^2=|df|^2$
\[\sA_t-\widetilde{A}=\frac{tA}{\sqrt{1+t^2A^2}}-\frac{A}{|A|}=\frac{A}{|A|}\left[\frac{t}{\sqrt{\frac{1}{|A|^2}+t^2\frac{A^2}{|A|^2}}}-1 \right]=\frac{A}{|A|}\left[\frac{t}{\sqrt{\frac{1}{|df|^2}+t^2}}-1 \right]\]
we have 
\[|\frac{t}{\sqrt{\frac{1}{|df|^2}+t^2}}-1|\leq\frac{1}{2t^2|df|^2}\]
therefore 

\[||(\sA_t-\widetilde{A})w||\leq ||\frac{1}{2t^2|df|^2} w||\leq \frac{||w||}{2t^2l^2} \xrightarrow{t\rightarrow \infty} 0\]
where we used $||\frac{w}{|df|^2}||\leq \frac{||w||}{l^2}$ which follows from the support assumption of $w$ and (\ref{bounddf}).

Now it remains to prove the estimate (\ref{mainestimate}). Without loss of generality, assume $r>0$; the case $r<0$ is the same. First, we have
\begin{equation}\label{relation1r}
    ||R(r,D_t)||\leq \frac{1}{r}
\end{equation}
   Indeed, by self-adjointness of $D_t$ we have for $u\in Dom (D_t)$
   \[||(D_t+ir)u||^2=||D_tu||^2+r^2||u||^2\geq r^2||u||^2\]
   thus
   \[||u||\leq \frac{1}{|r|}||(D_t+ir)u||\]
   $R(r,D_t)=(D_t+ir)^{-1}$ maps $L^2(\Omega^*(X_{reg}))$ to $Dom(D_t)$, hence for $v\in L^2(\Omega^*(X_{reg}))$
   \[||R(r,D_t)v||\leq \frac{1}{|r|}||(D_t+ir)(D_t+ir)^{-1}v||= \frac{1}{|r|}||v||\]
proving (\ref{relation1r}). We can write

\[R(r,tA)=(tA+ir)^{-1}=\frac{tA-ir}{t^2A^2+r^2}=\frac{tA-ir}{t^2|df|^2+r^2}\]
There are two cases. If $t\leq r$ then $0<t/r\leq1$ and
\begin{equation}\label{descrR}
    R(r,tA)w=\frac{1}{r}\left(\frac{(t/r)A-i}{(t/r)^2|df|^2+1}\right)w
\end{equation}
  Consider $$Q_{t,r}:=\left(\frac{(t/r)A-i}{(t/r)^2|df|^2+1}\right)=\left(\frac{(t/r)(df\wedge+\iota_{\nabla f})-i}{(t/r)^2|df|^2+1}\right)$$
  On $K\Subset X_{reg}$ this is a smooth bounded endomorphism of differential forms, and since $|df|>l>0$ and $df, \nabla f,\nabla^2f$ are all bounded, this implies that 
  \[\sup_{x\in K}||Q_{t,r}(x)|| \leq C, \quad \sup_{x\in K}|\nabla Q_{t,r}(x)|\leq C\]
  for some constant $C$. We have $supp(Dw)\subset supp (w)=K$. This implies
  \begin{equation}\label{qtr}
      ||D(Q_{t,r}w)||\leq C(||w||+||\nabla w||) =: C_w
  \end{equation}
for a constant $C_w$ depending on $w$ and independent of $t,r$ such that $0<t/r\leq1$. Therefore

\[||DR(r,tA)w||=||\frac{1}{r}DQ_{t,r}w||\leq \frac{C_w}{r}=\frac{C_w}{\max(t,r)}\]

   Similarly, if $t\geq  r$ then $0<r/t\leq1$ and

\[R(r,tA)=\frac{1}{t}\left(\frac{A-i(r/t)}{|df|^2+(r/t)^2}\right)=:\frac{1}{t}Q'_{t,r}\]
By the same argument there is a constant $C_w$ independent of $t,r$ satisfying $0<r/t\leq1$ such that
\begin{equation}\label{q'rt}
    ||DR(r,tA)w||=||\frac{1}{t}DQ'_{t,r}w||\leq \frac{C_w}{t}=\frac{C_w}{\max(t,r)}
\end{equation}
Therefore (\ref{relation1r}), (\ref{qtr}), and (\ref{q'rt}
) imply
\[ || R(r,D_t)DR(r,tA)w||\leq \frac{1}{r}\frac{C_w}{\max(t,r)}\]
proving (\ref{mainestimate}).
\end{proof}

\begin{lemma}\label{lemma2}
     For $g\in C(X)$ with $supp(g)\subset U$ we have as $t\rightarrow \infty$
     \begin{equation}
         ||[F_t,g]||\rightarrow 0
     \end{equation}
\end{lemma}
\begin{proof}
Again by boundedness of $F_t$, it suffices to show this for $g$ a Lipschitz function on $X$. We mentioned in the proof of Lemma \ref{lemm1} that the commutator $$C_g:=[D_t,g]=[D,g]$$ is bounded as a multiplication operator; moreover, it is supported in $supp(g)\subset U$. Consider a Lipschitz cutoff function $\psi$ such that 
\[\psi|_{supp(g)}=1,\quad supp(\psi)\subset U\]
Note that $C_g\psi=\psi C_g=C_g$.
Recall\[B_t=(1+D^2_t)^{\frac{-1}{2}}\]
We can apply Lemma \ref{lemm1} to get
\begin{equation}\label{psiB_t}
    ||\psi B_t||\xrightarrow{t\rightarrow \infty}0
\end{equation}
    We have 
   \[R(r,D_t)=(D_t+ir)^{-1}=B_tK_r(D_t) \]
   \[K_r(x):=\frac{\sqrt{1+x^2}}{x+ir}\]
   for $|r|>1$ we have $|K_r(x)| \leq1$. By the boundedness of $C_g,K_r(D_t)$ and (\ref{psiB_t}) we get
   \begin{equation}\label{commutc_gR}
       ||C_gR(r,D_t)||=||C_g\psi B_tK_r(D_t)||\xrightarrow{t\rightarrow \infty} 0
   \end{equation}
   Consider the following formal identity for an operator $T$
   \[[T^{-1},g]=-T^{-1}[T,g]T^{-1}\]
   Applying this to $T=D_t+ir, T^{-1}=R(r,D_t)$ (which is valid since $g$ preserves $Dom(D_t)$ and $C_g=[D_t,g]$ is bounded) we get
   \begin{equation}\label{comuutreolsventidentity}
       [R(r,D_t),g]=-R(r,D_t)[D_t,g]R(r,D_t)=-R(r,D_t)C_gR(r,D_t)
   \end{equation}
   Using (\ref{integralF_t}) we get (understood as a strong operator integral)
      \begin{align*}
       [F_t,g]=&\frac{1}{\pi}\int^{\infty}_0 \left([R(\sqrt{1+\lambda^2},D_t),g]+[R(-\sqrt{1+\lambda^2},D_t),g]\right) d\lambda
   \end{align*}
   So 
      \begin{align}\label{spectralftg}
       ||[F_t,g]||\leq&\frac{1}{\pi}\int^{\infty}_0 \left(||[R(\sqrt{1+\lambda^2},D_t),g]||+||[R(-\sqrt{1+\lambda^2},D_t),g]||\right) d\lambda
   \end{align}
We will use dominated convergence theorem to show that the integral in (\ref{spectralftg}) goes to zero as $t\rightarrow \infty$.  Recall from (\ref{relation1r})
   \begin{equation}\label{boundednessofR}
       ||R(r,D_t)||\leq \frac{1}{|r|}
   \end{equation}
By (\ref{commutc_gR}), (\ref{comuutreolsventidentity}), and (\ref{boundednessofR}), for a fixed $r$ with $|r|>1$ we have
\begin{equation}\label{comRggoesto0}
    ||[R(r,D_t),g]||\leq ||R(r,D_t)||.||C_gR(r,D_t)||\leq \frac{1}{|r|}||C_gR(r,D_t)||\xrightarrow{t\rightarrow \infty} 0
\end{equation}
thus the integrand in (\ref{spectralftg}) goes to zero pointwise  as $t\rightarrow \infty$. We need to show it's uniformly bounded by an integrable function. Applying (\ref{boundednessofR}) twice to (\ref{comuutreolsventidentity}) we get
   \[||[R(r,D_t),g]||\leq \frac{||C_g||}{r^2}\]
for $r=\pm\sqrt{1+\lambda^2}$
\[||[R(\pm\sqrt{1+\lambda^2},D_t),g]||\leq \frac{||C_g||}{1+\lambda^2}\]
  Since 
\[\int^{\infty}_0\frac{||C_g||}{1+\lambda^2}d\lambda \leq ||C_g|| \int^{\infty}_0\frac{1}{1+\lambda^2}d\lambda <\infty\]
  $ \frac{||C_g||}{1+\lambda^2}$ is integrable.
\end{proof}

\section{Restriction of $K$-homology to $U$}
We continue to drop $rel/abs$ from $D_{rel/abs}$ for simplicity of notation, and $D$ represents any of them.

Let $j:U\subset X$ be the inclusion. Then the extension by zero defines a homomorphism $$j^{\sharp}:C_0(U)\rightarrow C(X)$$ and thus induces a homomorphism \cite[\S~17.8]{B}
$$j^*:K^{an}_0(X)\rightarrow K^{an}_0(U)$$
This can be expressed in terms of the Kasparov product as follows. $j$ defines an element $$[j]\in KK_0(C_0(U),C(X))$$
and for $[x]\in K_0^{an}(X)=KK_0(C(X),\C)$
\[j^*([x])=[j]\otimes [x]\]
$j^*$ can be represented in terms of bounded $K$-cycles  simply by
\[(\mathcal{H},\Gamma, \rho_X, F)\mapsto (\mathcal{H},\Gamma, \rho_U,F)\]
i.e., we restrict the representation of $C(X)$ to $C_0(U)$.

Therefore $$j^*[D]=j^*[D_t]=[(L^2(\Omega^*(X_{reg})),\epsilon,C_0(U),F_t)] \in K^{an}_0(U)$$ 
Recall $\widetilde{A}=\frac{A}{|A|}$
\[\widetilde{A}^2=1,\quad \widetilde{A}^*=\widetilde{A},\quad [\widetilde{A},g]=0, ~ \forall g\in C_0(U)\]
Therefore the $K$-cycle represented by $\widetilde{A}$ is degenerate; thus its class in $K^{an}_0(U)$ vanishes.

\begin{theorem}
    As $t\rightarrow \infty$
\[(L^2(\Omega^*(X_{reg})),\epsilon,C_0(U),F_t) \rightarrow (L^2(\Omega^*(X_{reg})),\epsilon,C_0(U),\widetilde{A})\]
is a homotopy of Kasparov modules which implies $j^*[D]=0$.
\end{theorem}
\begin{proof}
This follows from the above lemmas similar to the proof given in \cite{R,LR}.  More precisely, to construct a \textit{standard homotopy} in the sense of \cite[\S17]{B},  consider the Hilbert $C([t_0,\infty])$-module $$\mathcal{E}=C\left([t_0,\infty],L^2(\Omega^*(X_{reg}))\right)$$ with $\mathcal{F}_t=F_t$ for $t<\infty$ and $\mathcal{F}_\infty=\widetilde{A}$. For $\mathcal{E}$ to define a Kasparov  $C_0(U)-C([t_0,\infty])$ module, we need to show (continuity and compactness at a finite $t$ are clear)

\begin{enumerate}
    \item $t\rightarrow \mathcal{F}_t$ is strong-operator continuous as $t\rightarrow \infty$.
    \item for a $g\in C_0(U)$, $t\rightarrow (\mathcal{F}^2_t-1)g$ is norm continuous as $t\rightarrow \infty$,  which in this case means it should go to zero.
    \item for a $g\in C_0(U)$,  $t\rightarrow [\mathcal{F}_t,g]$ is norm continuous as $t\rightarrow \infty$, which in this case also means it should go to zero.
\end{enumerate}
The first one is Lemma \ref{lemma3}, the second condition follows from the stronger result of Lemma  \ref{lemm1}, and the last one is Lemma  \ref{lemma2}. Lemmas \ref{lemm1}, \ref{lemma2} are stated for a $g\in C(X)$ with $supp(g)\subset U$, so they work for any $g\in C_c(U)$ by considering its  extension by zero. Density of $C_c(U)\subset C_0(U)$ and uniform boundedness of $||F_t||\leq1, ||(1+D^2_t)^{-1/2}||\leq 1$  imply the lemmas hold for $g\in C_0(U)$.
\end{proof}

\subsection{Proof of Theorem \ref{maintheorem}}
\begin{proof}
    We have proved $j^*[D]=0$. If there are no rel-critical points, then $U=X$, so $[D]=0$.
We may therefore assume that the critical set is nonempty. Let $i: \bigcup_r N_{\epsilon}(p_r) \hookrightarrow X$ be the inclusion. There is an exact sequence
\[K^{an}_0(\bigcup_r N_{\epsilon}(p_r))\xrightarrow{i_*} K^{an}_0(X) \xrightarrow{j^*} K^{an}_0(U)\]

$j^*[D]=0$ implies $$[D]\in im(K^{an}_0(\bigcup_r N_{\epsilon}(p_r))\xrightarrow{i_*} K^{an}_0(X))$$
Since each $N_{\epsilon}(p_r)$ is contractible to the critical point $\{p_r\}$ and they are pairwise disjoint, the homotopy invariance of the $K$-homology group implies
\[K^{an}_0(\bigcup_r N_{\epsilon}(p_r))\cong
\bigoplus_r K^{an}_0(\{p_r\})\cong \bigoplus_r \Z \]
Therefore we can write, for some $n_r\in \Z$
\[[D]=\sum_r n_ri_*[p_r]\]
Since $X$ is assumed to be connected (hence path connected as a Thom–Mather pseudomanifold is locally path connected), any two inclusions $\{p_r\}\hookrightarrow X$ and $\{p_s\}\hookrightarrow X$ are homotopic, so $i_*[p_r]=i_*[p_s]$, therefore
\[[D]=(\sum_r n_r)i_*[p_1]\in im(K^{an}_0(\{p_1\})\xrightarrow{} K^{an}_0(X))\]
\end{proof}

\end{document}